\documentclass[12pt]{amsart}

\newtheorem{theorem}{Theorem}[section]
\newtheorem{lemma}[theorem]{Lemma}
\newtheorem{proposition}[theorem]{Proposition}
\newtheorem{corollary}[theorem]{Corollary}

\theoremstyle{definition}

\newtheorem{remark}[theorem]{Remark}

\newenvironment{proofof}[1]{\noindent{\it Proof of
#1.}}{\hfill$\square$\\\mbox{}}

\def\Z{{\mathbb{Z}}}
\def\F{{\mathbb{F}}}
\def\N{{\mathbb{N}}}
\def\free{{\mathcal{F}}}
\def\zfree{{\mathcal{Z}}}
\def\bfree{{\mathcal{B}}}
\def\mfree{{\mathcal{M}}}
\def\conc{{C}}
\def\tr{{\mathrm{Tr}}}

\usepackage{amscd,amssymb}

\begin{document}

\title[Polynomial bound for the  nilpotency index of nil algebras]
{Polynomial bound for the nilpotency index of  finitely generated nil algebras}

\author[M. Domokos]
{M. Domokos}
\address{MTA Alfr\'ed R\'enyi Institute of Mathematics,
Re\'altanoda utca 13-15, 1053 Budapest, Hungary}
\email{domokos.matyas@renyi.mta.hu}

\thanks{This research was partially supported by National Research, Development and Innovation Office,  NKFIH K 119934.}

\subjclass[2010]{Primary: 16R10  Secondary: 16R30, 13A50, 15A72.}

\keywords{nil algebra, nilpotent algebra, matrix invariant, degree bound}

\begin{abstract} Working over an infinite field of positive characteristic, an upper 
bound is given for the nilpotency index of a finitely generated nil algebra of bounded nil index $n$ in terms of the maximal degree in a minimal homogenous generating system of the ring of simultaneous conjugation invariants of tuples of $n$ by $n$ matrices. This is deduced from a result of Zubkov. 
As a consequence, a recent degree bound due to Derksen and Makam for the generators of the ring of matrix invariants yields an upper  bound for the nilpotency index of a finitely generated nil algebra that is polynomial in the number of generators and the nil index. 
Furthermore, a characteristic free treatment is given to  Kuzmin's  
lower bound for the nilpotency index. 
\end{abstract}

\maketitle


\section{Introduction}\label{sec:intro}

Throughout this note $\F$ stands for  an infinite field of positive characteristic. All vector spaces, tensor products, algebras  are taken over $\F$.  
The results of this paper are valid in arbitrary characteristic, but they are known in characteristic zero (in fact stronger statements hold in characteristic zero, see Formanek \cite{formanek}, giving in particular an account of relevant works of Razmyslov \cite{razmyslov} and Procesi \cite{procesi}).  

Write 
$\free_m:=\F\langle x_1,\dots,x_m\rangle$ for the free associative $\F$-algebra with identity $1$ on $m$ generators $x_1,\dots,x_m$, 
and let $\free_m^+$ be its ideal generated by $x_1,\dots,x_m$ (so $\free_m^+$ is the free non-unitary associative algebra of rank $m$). 
For a positive integer $n$ denote by $I_{n,m}$ the ideal in $\free_m$ generated by $\{a^n\mid a\in \free_m^+\}$. 
A theorem of Kaplansky \cite{kaplansky} asserts that if a finitely generated associative algebra satisfies the polynomial identity $x^n=0$, then it is nilpotent. 
Equivalently, there exists a positive integer $d$ such that for all $i_1,\dots,i_d\in\{1,\dots,m\}$ the monomial  
$x_{i_1}\cdots x_{i_d}$ belongs to $I_{n,m}$. Denote by 
$d_{\F}(n,m)$ the minimal such $d$. In other words, $d_{\F}(n,m)$ is the minimal positive integer $d$ such that all $\F$-algebras that are generated by $m$ elements and satisfy the polynomial identity $x^n=0$ satisfy also the polynomial identity $y_1\cdots y_d=0$. 
This is a notable quantity of noncommutative ring theory: 
Jacobson \cite{jacobson} reduced the Kurosh problem for finitely generated algebraic algebras of bounded degree to the case of nil algebras of bounded degree. 
We mention also that proving nilpotency of nil rings under various conditions 
is a natural target for ring theorists, see for example the paper of Guralnick, Small and Zelmanov \cite{guralnick-small-zelmanov}.

The number $d_{\F}(n,m)$ is tightly connected with a quantity appearing in commutative invariant theory defined as follows. Consider the generic matrices 
\[X_r=(x_{ij}(r))_{1\le i.j\le n}, \quad r=1,\dots,m.\] 
These are elements in the algebra $A^{n\times n}$ of $n\times n$ matrices over the 
$mn^2$-variable commutative polynomial algebra 
$A=\F[x_{ij}(r)\mid 1\le i,j\le n,\ 1\le r\le m]$. 
The general linear group $GL_n(\F)$ acts on $A$ via $\F$-algebra automorphisms: for 
$g\in GL_n(\F)$ we have that $g\cdot x_{ij}(r)$ is the $(i,j)$-entry of the matrix 
$g^{-1}X_rg$. Set $R_{n,m}=A^{GL_n(\F)}$, the subalgebra of $GL_n(\F)$-invariants. 
This is the algebra of polynomial invariants under simultaneous conjugation of $m$-tuples of $n\times n$ matrices. The polynomial ring $A$ is graded in the standard way, and since the $GL_n(\F)$-action preserves the grading, the subalgebra $R_{n,m}$ is generated by homogeneous elements. Being the algebra of invariants of a reductive group, 
$R_{n,m}$ is finitely generated by the Hilbert-Nagata theorem (see for example \cite{newstead}).  We write $\beta_{\F}(n,m)$ for the minimal positive integer $d$ such that 
the $\F$-algebra $R_{n,m}$ is generated by elements of degree at most $d$. 
The main result of the present note is the following inequality: 

\begin{theorem}\label{thm:main} 
We have the inequality 
\[d_{\F}(n,m)\le \beta_{\F}(n,m+1).\]
\end{theorem} 

\begin{remark}\label{remark:mathpann} 
In the reverse direction it was shown in \cite[Theorem 3]{domokos} that for $n\ge 2$ we have 
\[\beta_{\F}(n,m)\le \lfloor \frac  n2\rfloor  d_{\F}(n,m).\] 
\end{remark} 

Theorem~\ref{thm:main} is derived from a  theorem of Zubkov \cite{zubkov} 
(for which Lopatin \cite{lopatin:jpaa} gave versions and improvements), 
see Theorem~\ref{thm:zubkov}. 
Using a result of Ivanyos, Qiao and Subrahmanyam \cite{ivanyos_etal:1}, Derksen and Makam \cite{derksen_makam:1} found strong bounds on the degrees of invariants defining the null-cone of $m$-tuples of $n\times n$ matrices 
under simultaneous conjugation, and derived from this the following upper bound on $\beta_{\F}(n,m)$: 

\begin{theorem}\label{thm:derksen_makam} (Derksen and Makam \cite[Theorem 1.4]{derksen_makam:2})  
We have the inequality 
\[\beta_{\F}(n,m)\le (m+1)n^4.\] 
\end{theorem} 

Given this result Derksen and Makam \cite[Conjecture 1.5]{derksen_makam:2} conjectured that there exists an upper bound on $d_{\F}(n,m)$ that is polynomial in $n$ and $m$. Combining Theorem~\ref{thm:main} and Theorem~\ref{thm:derksen_makam} we obtain the following affirmative answer to this conjecture: 

\begin{corollary}\label{cor:nilbound} 
We have the inequality 
\[d_{\F}(n,m)\le (m+2)n^4.\] 
\end{corollary}

\begin{remark}\label{remark:earlierbounds} Corollary~\ref{cor:nilbound} is a drastic improvement of the earlier known general upper bounds on $d_{\F}(n,m)$: 
\begin{enumerate}
\item  $d_{\F}(n,m)\le n^6m^{n+1}$ by Belov \cite{belov}.  
\item $d_{\F}(n,m)\le \frac 16 n^6m^{n}$ by Klein \cite{klein}. 
\item  $d_{\F}(n,m)\le 2^{18}mn^{12\log_3(n)+28}$ by Belov and Kharitonov \cite{belov_kharitonov}. 
\end{enumerate}
It is easy to see that $d_{\F}(2,m)\le m+1$. 
We note that for the case $n=3$ exact results on $d_{\F}(3,m)$ were obtained by Lopatin \cite{lopatin:commalg}. Moreover, Lopatin \cite{lopatin:jalg} proved that if $\mathrm{char}(\F)>\frac n2$ then $d_{\F}(n,m)\le n^{1+\log_2(3m+2)}$ and $d_{\F}(n,m)\le 2^{2+\frac n2}m$.  
\end{remark} 

\begin{remark} \label{remark:largechar}
When $\mathrm{char}(\F)>n^2+1$, we have $\beta_{\F}(n,m)\le n^2$.  
Indeed, the proof presented by Formanek \cite{formanek:86} (following the original arguments of  Razmyslov \cite{razmyslov} and Procesi \cite{procesi}) for the zero characteristic case of the corresponding inequality goes through without essential changes when $\mathrm{chat}(\F)>n^2+1$. 
Thus by Theorem~\ref{thm:main} we get that $d_{\F}(n,m)\le n^2$ when $\mathrm{char}(\F)>n^2+1$. 
\end{remark}

In Section~\ref{sec:lower} we  show that the following lower bound for $d_{\F}(n,m)$ due to   E. N. Kuzmin \cite{kuzmin} when $\mathrm{char}(\F)=0$ or 
$\mathrm{char}(\F)>n$ holds in arbitrary characteristic: 

\begin{theorem}\label{thm:kuzmin} 
The monomial $x_2x_1x_2x_1^2x_2x_1^3\cdots x_2x_1^{n-1}$ is not contained in the ideal $I_{n,2}$. 
In particular, for $m\ge 2$ we have $d_{\F}(n,m)\ge n(n+1)/2$. 
\end{theorem}

\begin{remark}\label{remark} It is well known that when $0<\mathrm{char}(\F)\le n$, the 
element $x_1x_2\cdots x_m$ is not contained in $I_{n,m}$, see for example 
\cite[5. Remarks. (I)]{nagata}. So in this case  for $m\ge 2$  we have 
\[\max\{m+1,n(n+1)/2\}\le d_{\F}(n,m)\le (m+2)n^4.\]  
\end{remark}


\section{Identities of matrices with forms}\label{sec:identities} 

The map $x_i\mapsto X_i$ $(i=1,\dots,m)$ extends to a unique $\F$-algebra homomorphism $\varphi_1:\free_m\to A^{n\times n}$. We have $\varphi_1(1)=I$, the $n\times n$ identity matrix. 
Consider the commutative polynomial algebra 
\[\mathcal{P}_{n,m}=\F[s_l(a)\mid a\in \free_m^+,\ l=1,\dots,n]\] 
generated by the infinitely many commuting indeterminates $s_l(a)$. 
Define  the $\F$-algebra homomorphism 
\[\varphi_2:\mathcal{P}_{n,m}\to R_{n,m}, \quad \varphi_2(s_l(a))=\sigma_l(\varphi_1(a))\] 
where for $B\in A^{n\times n}$ we have 
\[\det(tI+B)=\sum_{l=0}^nt^l\sigma_{n-l}(B),\] 
so $\sigma_l(B)$ is the sum of the principal $l\times l$ minors of $B$. 
A theorem of Donkin \cite{donkin} asserts that $\varphi_2$ is surjective onto $R_{n,m}$.  
Combining $\varphi_1$ and $\varphi_2$ we get an $\F$-algebra homomorphism 
\[\varphi:\mathcal{P}_{n,m}\otimes \free_m\to A^{n\times n}, \quad 
b\otimes a\mapsto \varphi_2(b)\varphi_1(a).\] 
The subalgebra $\conc_{n,m}=\varphi(\mathcal{P}_{n,m}\otimes \free_m)$ is called the {\it algebra of matrix concomitants}. It can be interpreted as the algebra of $GL_n(\F)$-equivariant polynomial maps $(\F^{n\times n})^m\to \F^{n\times n}$, where $GL_n(\F)$ acts on $\F^{n\times n}$  by conjugation and on the space $(\F^{n\times n})^m$ of $m$-tuples of matrices by simultaneous conjugation. 
For $a\in \free_m^+$ define an element $\chi_n(a)$ in $\mathcal{P}_{n,m}\otimes \free_m$ as follows: 
\[\chi_n(a)=\sum_{l=0}^n(-1)^{l}s_l(a)\otimes a^{n-l}\] 
(where $s_0(a)=1$). 
We need the following result of Zubkov \cite{zubkov} (see also Lopatin \cite[Theorem 2.4]{lopatin:jpaa}): 

\begin{theorem}\label{thm:zubkov} (Zubkov \cite{zubkov})  
The ideal $\ker(\varphi)$ is generated by 
\[\{b\otimes 1,\ \chi_n(a)\mid b\in\ker(\varphi_2),\ a\in \free_m^+\}.\]
\end{theorem} 

\begin{remark}
The papers \cite{zubkov} and \cite{lopatin:jpaa} use  different  commutative polynomial algebras than our $\mathcal{P}_{n,m}$, however, it is straightforward that Theorem~\ref{thm:zubkov} is an immediate consequence of the versions stated in \cite{zubkov}, \cite{lopatin:jpaa}.  We note that  \cite{zubkov}, \cite{lopatin:jpaa} give 
descriptions of the ideal $\ker(\varphi_2)$ as well. 
A self-contained approach to the theorem of Zubkov can be found in the recent book by 
De Concini and Procesi \cite{deconcini_procesi}. 
\end{remark} 

Denote by $\eta:\conc_{n,m}\to \conc_{n,m}/R_{n,m}^+\conc_{n,m}$ the natural surjection 
(ring homomorphism), where $R_{n,m}^+$ is the sum of the positive degree homogeneous components of $R_{n,m}$. 

\begin{corollary}\label{cor:1} 
The kernel of $\eta\circ \varphi_1$ is the ideal 
$I_{n,m}=(a^n\mid a\in \free_m^+)$ in $\free_m$. 
\end{corollary}

\begin{proof} We have $\ker(\eta\circ\varphi_1)=\ker(\eta\circ\varphi)\cap \free_m$ 
(where we identify $\free_m$ with the subalgebra $1\otimes \free_m$ in $\mathcal{P}_{n,m}\otimes \free_m$). 
The ideal $(s_l(a)\otimes 1 \mid a\in \free_m^+, \ 1\le l\le n)$ is mapped surjectively onto 
$R_{n,m}^+\conc_{n,m}$ by \cite{donkin}. Therefore we have 
\begin{align*}\ker(\eta\circ\varphi)=\varphi^{-1}(R_{n,m}^+\conc_{n,m})
 =\ker(\varphi)+(s_l(a)\otimes 1 \mid a\in \free_m^+, \ 1\le l\le n)
\\=(s_l(a)\otimes 1, 1\otimes  a^n\mid a\in\free_m^+, \ 1\le l\le n)  
\end{align*}
(the last equality follows from Theorem~\ref{thm:zubkov} and the fact that $1\otimes a^n-\chi_n(a)$ belongs to  $(s_l(a)\otimes 1 \mid a\in \free_m^+, \ 1\le l\le n))$. 
Obviously the ideal $(s_l(a)\otimes 1, 1\otimes a^n\mid a\in\free_m^+, \ 1\le l\le n)  $ intersects $\free_m$ in $I_{n,m}$. 
\end{proof} 

\begin{remark} Corollary~\ref{cor:1} implies that the relatively free algebra 
$\free_m/I_{n,m}$ is isomorphic to $\conc_{n,m}/R_{n,m}^+\conc_{n,m}$. 
When $\mathrm{char}(\F)=0$, this statement is due to Procesi \cite[Corollary 4.7]{procesi}. 
\end{remark}

The algebras $R_{n,m}$ and $\conc_{n,m}$ are $\Z^m$-graded: 
\[\deg_m(X_{i_1}\cdots X_{i_d})=(\alpha_1,\dots,\alpha_m) \text{ where }
\alpha_k=|\{j\mid i_j=k\}|\] 
and 
\[\deg_m(\sigma_l(X_{i_1}\cdots X_{i_d})) =l\cdot \deg_m(X_{i_1}\cdots X_{i_d}).\] 

\bigskip 

\begin{proofof}{Theorem~\ref{thm:main}} Set $d=\beta_{\F}(n,m+1)$. We have to show that 
$x_{i_1}\cdots x_{i_d}\in I_{n,m}$ for all $i_1,\dots,i_d\in\{1,\dots,m\}$. 
Recall that by \cite{donkin} the algebra $R_{n,m+1}$ is generated by the elements $\sigma_l(W)$, where $W$ is a word in $X_1,\dots,X_{m+1}$, and $l\in\{1,\dots,n\}$. 
The total degree of the element  $\tr(X_{i_1}\cdots X_{i_d}X_{m+1})\in R_{n,m+1}$ is strictly greater than $\beta_{\F}(n,m+1)$, whence we have a relation 
\begin{equation}\label{eq:longtrace}
\tr(X_{i_1}\cdots X_{i_d}X_{m+1})=\sum_{\lambda\in\Lambda} a_{\lambda} f_{\lambda}
\end{equation}
where $\Lambda$ is a finite index set, $a_{\lambda}\in\F$, and  each 
$f_{\lambda}\in R_{n,m+1}$ is a product $f_{\lambda}=\sigma_{l_1}(W_1)\cdots \sigma_{l_r}(W_r)$ with $r\ge 2$ and $W_1,\dots,W_r$ non-empty words in $X_1,\dots,X_{m+1}$. 
The $\Z^{m+1}$-multidegree of $\tr(X_{i_1}\cdots X_{i_d}X_{m+1})$ is 
\[\deg_{m+1}(\tr(X_{i_1}\cdots X_{i_d}X_{m+1}))=(\deg_m(\tr(X_{i_1}\cdots X_{i_d})),1).\] 
The terms $f_{\lambda}$ are all $\Z^{m+1}$-homogeneous, whence we may assume that each has the above $\Z^{m+1}$-degree (since the other possible terms on the right hand side of \eqref{eq:longtrace} must cancel each other). 
It follows that for each $f_{\lambda}$ exactly one of its factors 
$\sigma_{l_1}(W_1), \dots, \sigma_{l_r}(W_r)$ has $\Z^{m+1}$-degree of the form 
$(\alpha_1,\dots,\alpha_m,1)$, say this is $\sigma_{l_1}(W_1)$, and the remaining factors have $\Z^{m+1}$-degree 
of the form $(\gamma_1,\dots,\gamma_m,0)$. 
Necessarily we have $l_1=1$ and so $\sigma_{l_1}(W_1)=\tr(X_{m+1}Z)$ for some (possibly empty) word $Z$ in 
$X_1,\dots,X_m$, and $W_2,\dots,W_r$ are non-empty words in $X_1,\dots,X_m$. Set 
\[g_{\lambda}=\sigma_{l_2}(W_2)\cdots \sigma_{l_r}(W_r)Z\in\conc_{n,m}, \]
and note that $f_{\lambda}=\tr(g_{\lambda}X_{m+1})$.  
Using linearity of $\tr(-)$ relation \eqref{eq:longtrace} can be written as 
\begin{equation}\label{eq:2longtrace}
\tr(X_{m+1}(X_{i_1}\cdots X_{i_d}-\sum_{\lambda\in\Lambda} a_{\lambda} g_{\lambda}))=0\in R_{n,m+1}.\end{equation} 
Substituting $X_{m+1}\mapsto E_{ij}$ (the matrix whose $(i,j)$-entry is $1$ and all other entries are $0$) we get from \eqref{eq:2longtrace} that the $(j,i)$-entry of 
$X_{i_1}\cdots X_{i_d}-\sum_{\lambda\in \Lambda} a_{\lambda}g_{\lambda}$ is $0$. 
This holds for all $(i,j)$, thus we have the equality 
\begin{equation}\label{eq:3longtrace}X_{i_1}\cdots X_{i_d}=\sum_{\lambda}a_{\lambda}g_{\lambda}.
\end{equation}  
The right hand side of \eqref{eq:3longtrace} is obviously contained in 
$R_{n,m}^+\conc_{n,m}$, therefore 
it follows from \eqref{eq:3longtrace} that the element $x_{i_1}\cdots x_{i_d}\in\free_m$ belongs to the kernel of $\eta\circ \varphi_1$.  Thus by Corollary~\ref{cor:1} we conclude 
that $x_{i_1}\cdots x_{i_d}\in I_{n,m}$.  
\end{proofof}  


\section{Lower bound}\label{sec:lower} 

Kuzmin's proof of the case $\mathrm{char}(\F)=0$ or 
$\mathrm{char}(\F)>n$ of Theorem~\ref{thm:kuzmin} 
(it is presented also in the survey of Drensky in \cite{drensky-formanek}) 
uses crucially Lemma~\ref{lemma:linearization} below, relating 
the {\it complete linearization of} $x^n$, namely 
\[P_n(x_1,\dots,x_n)=\sum_{\pi\in \mathrm{Sym}\{1,\dots,n\}}x_{\pi(1)}x_{\pi(2)}\cdots  
x_{\pi(n)}\in \free_n.\]

\begin{lemma}\label{lemma:linearization} 
If $\mathrm{char}(\F)=0$ or $\mathrm{char}(\F)>n$, 
then $I_{n,m}$ is spanned as an $\F$-vector space by the elements 
$P_n(w_1,\dots,w_n)$, where $w_1,\dots,w_n$ range over all non-empty monomials in $x_1,\dots,x_m$. 
\end{lemma} 

\begin{remark}\label{remark:characteristic}
The assumption on $\mathrm{char}(\F)$ in Lemma~\ref{lemma:linearization} is necessary, its statement obviously fails if $0<\mathrm{char}(\F)\le n$ 
(as it can be easily seen already in the special case $m=1$).  
Now we modify the arguments of Kuzmin to obtain Theorem~\ref{thm:kuzmin} in a characteristic free manner. It turns out that although Lemma~\ref{lemma:linearization} can not be applied, the main combinatorial ideas of Kuzmin's proof do work. 
\end{remark} 

Consider the free $\Z$-algebra $\zfree=\Z \langle x,y\rangle^+$ without unity. Write $\mfree$ for the set of non-empty monomials (words) in $x,y$. 
For a positive integer $k$ write $\zfree(k)$ for the $\Z$-submodule of $\zfree$ generated by the $w\in \mfree$  whose total degree in $y$ is $k-1$. 
It will be convenient to use the following notation: for $(a_1,\dots,a_k)\in \N_0^k$ set 
\[[a_1,\dots,a_k]=x^{a_1}yx^{a_2}y\cdots yx^{a_k}\in \mfree.\] 
The symmetric group $S_k=\mathrm{Sym}\{1,\dots,k\}$ acts on the right linearly on 
$\zfree(k)$, extending linearly the permutation action on $\zfree(k)\cap \mfree$ given by 
\[[a_1,\dots,a_k]^{\pi}=[a_{\pi(1)},\dots,a_{\pi(k)}] \quad \text{ for }\pi\in S_k.\]   
Let $\bfree$ denote the $\Z$-submodule of $\zfree$ generated by all the elements 
$[a_1,\dots,a_k]$ ($k\in \N$) such that $a_i\ge n$ for some $i\in\{1,\dots,k\}$ or $a_i=a_j$ for some 
$1\le i<j\le k$, and by all the elements of the form 
$[a_1,\dots,a_k]+[a_1,\dots,a_k]^{(ij)}$ where $(ij)$ denotes the transposition interchanging 
$i$ and $j$ for  $1\le i<j\le k$.  We shall use the following obvious properties of $\bfree$: 

\begin{lemma}\label{lemma:bproperties}
\begin{itemize}
\item[(i)] The $\Z$-submodule $\bfree\cap \zfree(k)$ of $\zfree(k)$ is  $S_k$-stable.  
\item[(ii)] We have the inclusions 
$y\bfree\subset \bfree$, $\zfree y\bfree\subset \bfree$, $\bfree y\subset \bfree$, 
and $\bfree y\zfree \subset \bfree$. 
\item[(iii)] Let $k$ be a positive integer, $u_1,\dots,u_{k-1}\in\mfree $ monomials such that 
$u_i\in y\zfree\cap \zfree y$ or $u_i=y$ for $i=1,\dots,k-1$. Then $\bfree$ contains the image of the $\Z$-module map on $\bfree\cap \zfree(k)$ given by 
\[[a_1,\dots,a_k]\mapsto x^{a_1}u_1x^{a_2}u_2 x^{a_3}\cdots u_{k-1}x^{a_k}.\] 
\item[(iv)] For any positive integer $a$, the $\Z$-submodule $\bfree$ of $\zfree$ is preserved by the derivation $\delta_a$ on $\zfree$ defined by $\delta_a(x)=x^a$, $\delta_a(y)=0$. 
\item[(v)] The factor $\zfree/\bfree$ is a free $\Z$-module freely generated by the images under the natural surjection $\zfree\to \zfree/\bfree$ of the monomials 
\[\widehat\mfree=\{[a_1,\dots,a_k]\mid k\in \mathbb{N},\quad 0\le a_1<a_2<\dots<a_k\le n-1\}.\] 
\end{itemize}
\end{lemma} 
\begin{proof} Statements (i), (ii), (iii), (iv) are immediate consequences of the construction of $\bfree$. To prove (v) note that $\zfree=\bigoplus \zfree(c_1,\dots,c_k)$ where the direct sum is taken over 
$k\in \N$ and $0\le c_1\le \cdots \le c_k$, and 
$\zfree(c_1,\dots,c_k)$ stands for the $\Z$-submodule generated by $[c_1,\dots,c_k]^\pi$ 
as $\pi$ ranges over $S_k$. Moreover, 
$\bfree=\bigoplus \bfree(c_1,\dots,c_k)$ 
where  $\bfree(c_1,\dots,c_k)=\bfree\cap  \zfree(c_1,\dots,c_k)$. 
Now $\zfree(c_1,\dots,c_k)\subset \bfree$ if $c_i=c_j$ for some $i\neq j$ or if $c_i\ge n$ for some $i$. It is also clear that for $0\le a_1<\cdots<a_k$ we have 
$\zfree(a_1,\dots,a_k)=\Z \cdot [a_1,\dots,a_k]+\bfree(a_1,\dots,a_k)$, so 
the monomials in $\widehat\mfree$ generate the $\Z$-module $\zfree $ modulo $\bfree$.  
Suppose that some non-trivial $\Z$-linear combination of the elements in $\widehat\mfree$ 
belongs to $\bfree$. The above direct sum decompositions of  $\zfree$ and $\bfree$ imply then that there exist $q,k\in \N$, and $0\le a_1<\cdots<a_k\le n-1$ such that  
$q[a_1,\dots,a_k]\in \bfree(a_1,\dots,a_k)$. This means that 
\begin{align}\label{eq:epsilon} 
q[a_1,\dots,a_k]=\sum_{i=1}^s\varepsilon_i (w_i+w_i^{\pi_i})
\end{align}  where 
$\varepsilon_i=\pm 1$, $w_i\in \zfree(a_1,\dots,a_k)\cap \mfree$ and $\pi_i\in S_k$ is a transposition for $i=1,\dots,s$. Suppose that $s$ in \eqref{eq:epsilon} is minimal possible. Without loss of generality we may assume that $w_1=[a_1,\dots,a_k]$ and $\varepsilon_1=1$. The word $w_1^{\pi_1}$ must be canceled by  some summand $\varepsilon_i(w_i+w_i^{\pi_i})$ with $i\ge 2$ on the right hand side of \eqref{eq:epsilon}, so after a possible renumbering we have $\varepsilon_2(w_2+w_2^{\pi_2})=-(w_1^{\pi_1}+w_1^{\pi_1\pi_2})$. 
Now the term $-w_1^{\pi_1\pi_2}$ must be canceled by $w_1$ or by some summand  
$\varepsilon_i(w_i+w_i^{\pi_i})$ with $i\ge 3$. It means that the right hand side of 
\eqref{eq:epsilon} has a subsum of the form 
\begin{align}\label{eq:subsum}
(w_1+w_1^{\pi_1})-(w_{1}^{\pi_1}+w_1^{\pi_1\pi_2})+(w_1^{\pi_1\pi_2}+w_1^{\pi_1\pi_2\pi_3})-+\cdots+(-1)^{r-1}(w_1^{\pi_1\cdots\pi_{r-1}}+
w_1^{\pi_1\cdots \pi_r})
\end{align} 
where $w_1^{\pi_1\cdots\pi_r}=w_1$. This latter equality forces that $\pi_1\cdots \pi_r$ is the identity permutation, so $r$ is even, and then the sum \eqref{eq:subsum} is zero. 
So all these terms can be omitted from \eqref{eq:epsilon}. This contradicts the minimality of $s$. 
This shows that $q[a_1,\dots,a_k]$ is not contained in $\bfree$. 
\end{proof}

\begin{lemma}\label{lemma:a1} 
Let $k$ be a positive integer, $a_1\le a_2\le \dots\le a_k\in \N_0$, and $r\in\N_0$ with 
$a_1+k+r>n$. 
Then 
\begin{equation}\label{eq:sumsum}
\sum_{c_1+\cdots+c_k=r}\sum_{\pi\in S_k}[a_1+c_{\pi(1)},\dots,a_k+c_{\pi(k)}]\in\bfree.
\end{equation}  
\end{lemma} 

\begin{proof} Apply induction on $k$. In the case $k=1$ the element in question in \eqref{eq:sumsum} is $x^{a_1+r}$, which belongs to $\bfree$ by the assumption 
$a_1+1+r>n$.  Suppose next that $k>1$, and the statement of the lemma holds for smaller $k$.  
The terms $[a_1+d_1,\dots,a_k+d_k]$ 
in the sum \eqref{eq:sumsum} can be grouped into three classes: 
\begin{enumerate} 
\item[(A)] $a_1+d_1<a_2$
\item[(B)] $a_1+d_1=a_2+d_2$
\item[(C)] $a_1+d_1\ge a_2$ and $a_1+d_1\neq a_2+d_2$. 
\end{enumerate} 
The sum of the terms of type (A) is a sum of expressions of the form 
\begin{equation}\label{eq:typeA}
x^{a_1+d_1}y\sum_{c_2+\cdots+c_k=r-d_1}\sum_{\pi\in\mathrm{Sym}\{2,\dots,k\}}
[a_2+c_{\pi(2)},\dots,a_k+c_{\pi(k)}].\end{equation} 
Here $a_2+(k-1)+(r-d_1)\ge a_1+k+r >n$, hence by the induction hypothesis 
$\sum_{c_2+\cdots+c_k=r-d_1}\sum_{\pi\in\mathrm{Sym}\{2,\dots,k\}}
[a_2+c_{\pi(2)},\dots,a_k+c_{\pi(k)}]$ belongs to $\bfree$. 
Now by Lemma~\ref{lemma:bproperties} (ii) we conclude that the element in \eqref{eq:typeA} belongs to $\bfree$. 
The terms of type (B) belong to $\bfree$ by construction of $\bfree$. 
Finally, a term $[a_1+d_1,\dots,a_k+d_k]$ of type (C) can be paired off with the term 
$[a_1+e_1,a_2+e_2,a_3+d_3,\dots,a_k+d_k]$ where  
$e_1=a_2-a_1+d_2$ and $e_2=a_1-a_2+d_1$ (so this is also of type  (C)), and the sum of these two terms belongs to $\bfree$ by construction of $\bfree$. 
\end{proof} 

\begin{corollary}\label{cor:r+k}
Let $k$ be a positive integer, $(a_1,\dots,a_k)\in \N_0^k$, and $r\in \N_0$ with $r+k>n$. 
Then 
\[\sum_{c_1+\cdots+c_k=r}\sum_{\pi\in S_k}[a_1+c_{\pi(1)},\dots,a_k+c_{\pi(k)}]\in\bfree.\] 
\end{corollary}
\begin{proof} 
Take a permutation $\rho\in S_k$ such that $a_{\rho(1)}\le \dots \le a_{\rho(k)}$. 
Applying $\rho$ to the element in the statement we get 
\[\sum_{c_1+\cdots+c_k=r}\sum_{\pi\in S_k}[a_{\rho(1)}+c_{\pi(1)},\dots,a_{\rho(k)}+c_{\pi(k)}],\] 
which belongs to $\bfree\cap \zfree(k)$ by Lemma~\ref{lemma:a1}. 
Our statement follows  by Lemma~\ref{lemma:bproperties} (i).  
\end{proof} 

\begin{lemma}\label{lemma:xxx} 
Suppose $1\le k\le n+1$, $w_1,\dots,w_{k-1}\in\mfree$ are monomials having positive degree in $y$, and $a,b\in\N_0$. 
Then 
\begin{equation} \label{eq:xxx} 
x^aP_n(w_1,\dots,w_{k-1},x,\dots,x)x^b\in\bfree.
\end{equation}   
\end{lemma} 

\begin{proof}
We have $w_i=x^{a_i}u_ix^{b_i}$ where $a_i,b_i\in\N_0$ and $u_i\in y\zfree\cap \zfree y$ or $u_i=y$ ($i=1,\dots,k-1$). Then the element in \eqref{eq:xxx} is 
\[\sum_{\rho\in S_{k-1}}\left((n-k+1)!\sum_{c_1+\cdots+c_k=n-k+1}\sum_{\pi\in S_k}
x^{d_1+c_{\pi(1)}}u_{\rho(1)}x^{d_2+c_{\pi(2)}}u_{\rho(2)}\cdots x^{d_{k-1}+c_{\pi(k-1)}}u_{\rho(k-1)}x^{d_k+c_{\pi(k)} }\right)\]
where $d_1=a+a_{\rho(1)}$, $d_2=a_{\rho(2)}+b_{\rho(1)}$, 
$d_3=a_{\rho(3)}+b_{\rho(2)}$, $d_{k-1}=a_{\rho(k-1)}+b_{\rho(k-2)}$,  
$d_k=b_{\rho(k-1)}+b$. 
The summand  corresponding to $\rho\in S_{k-1}$ in the outer sum is contained in $\bfree$ by Corollary~\ref{cor:r+k} and Lemma~\ref{lemma:bproperties} (iii). 
\end{proof}

\begin{lemma}\label{lemma:w} 
For any $w_1,\dots,w_n\in\mfree$, $w_0,w_{n+1}\in\mfree\cup \{1\}$  we have 
\[w_0P_n(w_1,\dots,w_n)w_{n+1}\in\bfree.\]  
\end{lemma}

\begin{proof} By Lemma~\ref{lemma:bproperties} (ii) it is sufficient to deal with the case 
$w_0=x^a$, $w_{n+1}=x^b$. We may assume that $w_1,\dots,w_{k-1}$ have positive degree in $y$ and $w_{k-1+j}=x^{c_j}$ for $j=1,\dots,n-k+1$. If $n-k+1=0$ or all the $c_j=1$ then we are done by Lemma~\ref{lemma:xxx}. Suppose next  that $n-k+1>0$, $c_1,\dots,c_l>1$ with $l\ge 1$, and $c_{l+1}=\cdots =c_{n-k+1}=1$. By induction on $l$ we show that 
$x^aP_n(w_1,\dots,w_{k-1},x^{c_1},\dots,x^{c_l},x,\dots,x)x^b\in\bfree$. 
By the induction hypothesis (or by Lemma~\ref{lemma:xxx} when $l=1$) $f=x^aP_n(w_1,\dots,w_{k-1},x^{c_1},\dots,x^{c_{l-1}},x,\dots,x)x^b\in\bfree$, hence by Lemma~\ref{lemma:bproperties} (iv) 
$\delta_{c_l}(f)\in \bfree$. 
We have 
\begin{align*}
\delta_{c_l}(f)&=
ax^{a+c_l-1}P_n(w_1,\dots,w_{k-1},x^{c_1},\dots,x^{c_{l-1}},x,\dots,x)x^b
\\ &+\sum_{i=1}^{k-1}x^aP_n(w_1,\dots,\delta_{c_l}(w_i),\dots,w_{k-1},x^{c_1},\dots,x^{c_{l-1}},x,\dots,x)x^b
\\&+ \sum_{j=1}^{l-1}c_jx^aP_n(w_1,\dots,w_{k-1},x^{c_1},\dots,x^{c_j+c_l-1}, \dots, x^{c_{l-1}},x,\dots,x)x^b
\\ &+(n-k-l+2)x^aP_n(w_1,\dots,w_{k-1},x^{c_1},\dots,x^{c_l},x,\dots,x)x^b
\\ &+bx^aP_n(w_1,\dots,w_{k-1},x^{c_1},\dots,x^{c_{l-1}},x,\dots,x)x^{b+c_l-1}. 
\end{align*} 
All other terms than 
$(n-k-l+2)x^aP_n(w_1,\dots,w_{k-1},x^{c_1},\dots,x^{c_l},x,\dots,x)x^b$ 
on the right hand side above belong to $\bfree$ by the induction hypothesis. 
Taking into account that $\zfree/\bfree$ is torsion free by Lemma~\ref{lemma:bproperties} (v) 
we conclude the desired inclusion 
\[x^aP_n(w_1,\dots,w_{k-1},x^{c_1},\dots,x^{c_l},x,\dots,x)x^b\in \bfree.\] 
\end{proof} 

For $\lambda=(\lambda_1,\dots,\lambda_m)\in \N_0^m$ denote by  
$P_{\lambda}(x_1,\dots,x_m)\in \Z\langle x_1,\dots,x_m\rangle$ 
the multihomogeneous component of $(x_1+\cdots+x_m)^n$ having $\Z^m$-degree 
$\lambda$. 

\begin{corollary}\label{cor:lambda} 
For any $m\in\N$, $w_1,\dots,w_m\in \mfree$, $w_0,w_{m+1}\in \mfree\cup\{1\}$ and for 
any $\lambda\in \N_0^m$ 
we have that 
\[w_0P_{\lambda}(w_1,\dots,w_m)w_{m+1}\in\bfree.\]  
\end{corollary}

\begin{proof}
We have the equality 
\[P_{\lambda}(x_1,\dots,x_m)=\frac{1}{\prod_{i=1}^m(\lambda_i!) }
P_n(\underbrace{x_1,\dots,x_1}_{\lambda_1},\dots,\underbrace{x_m,\dots,x_m}_{\lambda_m}).\] 
Therefore the statement follows from Lemma~\ref{lemma:w} by Lemma~\ref{lemma:bproperties} (v). 
\end{proof} 

\begin{proposition}\label{prop:B} 
The ideal $I_{n,2}$ is contained in the subspace $\F\otimes_{\Z}\bfree$ of 
$\F\langle x,y\rangle$. 
\end{proposition} 

\begin{proof}
The ideal $I_{n,2}$ is spanned as an $\F$-vector space by 
elements of the form 
\[w_0(c_1w_1+\cdots+c_mw_m)^nw_{m+1},\] 
where 
the $w_i$ are monomials in $x,y$ and they have positive total degree for $i=1,\dots,m$, and $c_1,\dots,c_m\in \F$.   
Since we have the equality 
\[(c_1w_1+\cdots+c_mw_m)^n=\sum_{\lambda\in\N_0^m, \ \lambda_1+\cdots+\lambda_m=n}c_1^{\lambda_1}\cdots c_m^{\lambda_m}P_{\lambda}(w_1,\dots,w_m),\] 
our statement follows from Corollary~\ref{cor:lambda}. 
\end{proof} 

\bigskip
\begin{proofof}{Theorem~\ref{thm:kuzmin}} 
By Lemma~\ref{lemma:bproperties} (v) the monomials 
\[\{x^{a_1}yx^{a_2}yx^{a_3}\cdots yx^{a_k}\mid 0\le a_1<a_2<\cdots<a_k\le n-1\}\] 
are linearly independent in $\free_2=\F\langle x,y\rangle$ modulo the subspace 
$\F\otimes_{\Z}\bfree$. Since $\F\otimes_{\Z}\bfree$ contains the ideal $I_{n,2}$ by 
Proposition~\ref{prop:B}, our statement follows.  
\end{proofof}

\end{document}